\documentclass[11pt]{amsart}
\usepackage{geometry}                
\usepackage{graphicx}
\usepackage{amssymb}
\usepackage{epstopdf}
\usepackage{amsthm}
\usepackage{graphicx}
\usepackage{tikz-cd}
 \usepackage{booktabs}
 \usepackage{hyperref}

\newtheorem{theorem}{Theorem}[section]
\newtheorem{definition}[theorem]{Definition}
\newtheorem{question}[theorem]{Question}
\newtheorem{lemma}[theorem]{Lemma}
\newtheorem{example}[theorem]{Example}
\newtheorem{proposition}[theorem]{Proposition}
\newtheorem{cor}[theorem]{Corollary}
\newtheorem{conjecture}[theorem]{Conjecture}

\newtheorem*{theorem*}{Theorem}
\newtheorem*{proposition*}{Proposition}
\newtheorem*{cor*}{Corollary}
\newtheorem*{conjecture*}{Conjecture}

\newcommand{\abk}[1]{\left\langle #1 \right\rangle}
\newcommand{\C}{\mathcal C}
\newcommand{\R}{\mathbb R}
\newcommand{\N}{\mathbb N}

\newcommand{\lk}{\mathrm {Link}}
\newcommand{\pierce}[4]{\mathrm{pierce}_{( { #2}, {#3}, {#4})}({#1}) }
\newcommand{\code}{\mathrm{code}}

\newcommand{\suspect}[1]{ }
\newcommand{\revision}[1]{#1}

\graphicspath{{/Users/caitlinlienkaemper/Documents/Academic/figures/}}

\title{Geometry, combinatorics, and algebra of inductively pierced codes} 
\author{Caitlin Lienkaemper}

\date{\today}                 

\begin{document}
 \begin{abstract}
Convex neural codes are combinatorial structures describing the intersection pattern of a collection of convex sets.  
Inductively pierced codes are a particularly nice subclass of neural codes introduced in the information visualization literature by Stapleton et al. in 2011 and to the convex codes literature by Gross et al. in 2016. Here, we show that all inductively pierced codes are nondegenerate convex codes and nondegenerate hyperplane codes. In particular, we prove that a $k$-inductively pierced code on $n$ neurons has a convex realization with balls in $\R^{k+1}$ and with half spaces in $\R^{n}$. We characterize the simplicial and polar complexes of inductively pierced codes, showing the simplicial complexes are disjoint unions of vertex decomposable clique complexes and that the polar complexes are shellable. In an earlier version of this preprint, we gave a flawed proof that toric ideals of $k$-inductively pierced codes have quadratic Gr\"obner bases under the term order induced by a shelling order of the polar complex of $\C$. We now state this as a conjecture, inspired by computational evidence.
\end{abstract} 
\maketitle

\tableofcontents
\section{Introduction}
Combinatorial neural codes describe the activity of a population of neurons in terms of which neurons fire together and which do not. A number of authors have sought to characterize the class of convex neural codes, neural codes which have an intrinsic discrete-geometric structure because they arise from the activity of neurons with convex receptive fields  in some space of stimuli \cite{curto2013neural, curto2017makes, lienkaemper2017obstructions, jeffs2018morphisms}.   For instance, hippocampal place cells have receptive fields which correspond to convex regions in an animal's environment. When the animal is in the receptive field of a given place cell, that cell will fire. Thus, neurons fire together if and only if their receptive fields overlap.  

\begin{figure}[!h]
\includegraphics[width = 4 in]{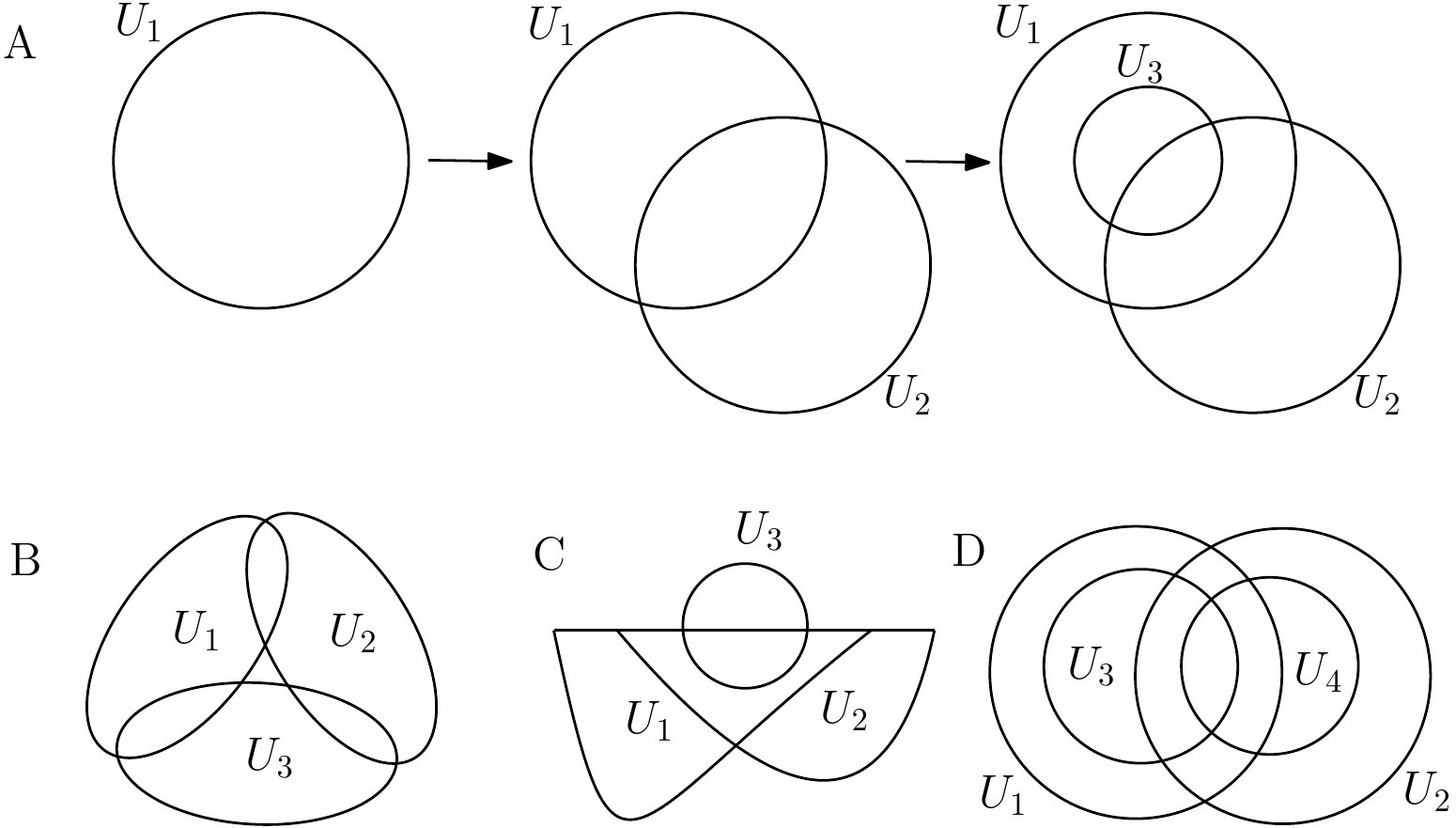}
\caption{\label{fig:main_intro} Examples and non-examples of inductively pierced codes. (A) Building up the code $\{\varnothing, 1, 12, 2, 123, 13\}$ inductively. (B-D) Three convex realizations of neural codes which are not inductively pierced. }
\end{figure}

In order to characterize convex codes intrinsically, without reference to any particular realization, we study a variety of combinatorial and algebraic signatures of convex and non-convex codes. For instance, a code is \emph{max-intersection-complete} if it contains all intersections of its maximal codewords and is \emph{intersection-complete} if it contains all  intersections of its codewords. All max-intersection-complete codes, and therefore all intersection-complete codes, are convex \cite{cruz2016open}. Other combinatorial signatures of convexity and non-convexity stem from the structure of the \emph{simipicial complex} of a code, defined to be the smallest simplicial complex containing the code \cite{curto2017makes, lienkaemper2017obstructions}. Therefore, given a class of codes, it makes sense to ask about the structure of their simplicial complexes. For instance, are they contractible? Collapsible? What about the \emph{polar complex} of a code, a simplicial complex whose facets are in bijection with the codewords? 

These combinatorial signatures and others are reflected in the structure of the \emph{neural ideal} $J_\C$, a \emph{pseudo-monomial} ideal which encodes the combinatorial data of the neural code. In particular, the \emph{canonical form} $CF(J_\C)$ of a neural ideal uses the set of pseudo-monomials which are minimal with respect to divisibility to encode the minimal receptive field relationships in an algebraic way. Various combinatorial signatures of convexity or non-convexity, such as intersection-completeness, show up as algebraic signatures in $CF(J_\C)$. 
 
Inductively $k$-pierced codes are a restrictive class of neural codes inspired by work in information visualization \cite{stapleton2011drawing} and introduced in the context of convex neural codes by Gross, Obatake, and Youngs in \cite{gross2016neural}. Some examples and non-examples of realizations of inductively pierced codes are given in Figure \ref{fig:main_intro}.  Building on \cite{gross2016neural}, we use algebraic and combinatorial invariants to achieve a deeper understanding of the class of inductively pierced codes, answering the questions posed in this introduction. Beyond this, we explore the neural toric ideal of an inductively pierced code, making progress on a conjecture of \cite{gross2016neural}.

We begin by building on \cite{gross2016neural} to establish the basic geometric and algebraic properties of inductively pierced codes. In particular, we show that inductively pierced codes always have realizations resembling that in Figure \ref{fig:main_intro} A:
\begin{proposition*}[Proposition \ref{prop:realization}]
An inductively $k$-pierced code has a convex realization by open balls of dimension $k+1$. In particular, the minimal embedding dimension of an inductively $k$-pierced code is at most $k+1$.
\end{proposition*}
This generalizes a result in the information visualization literature that inductively 1-pierced codes and some inductively 2-pierced codes can be realized with disks in $\R^2$  \cite{stapleton2011drawing}.  

Next, we pull out some immediate combinatorial corollaries of the full characterization of the neural ideals and canonical forms of inductively pierced codes given in \cite{gross2016neural}. In particular, we show that inductively pierced codes are intersection-complete and that their simplicial complexes are clique complexes.

Next, we show that the simplicial complex of an inductively pierced code is a disjoint union of \emph{vertex-decomposable} complexes. Vertex-decomposability is a strong combinatorial condition on simplicial complexes which implies collapsibility and, in turn, contractibility. 
\begin{theorem*}[Theorem \ref{thm:vertex_decomp}]The simplicial complex of an inductively pierced code is a disjoint union of vertex decomposable simplicial complexes. 
\end{theorem*}

We discuss the polar complexes of inductively pierced codes, which are simplicial complexes on $2n$ vertices whose facets correspond to codewords. We show that polar complexes of inductively pierced codes are shellable. Roughly, this means that there is a way to order their facets such that the complex can be glued together nicely.
   \begin{theorem*}[Theorem \ref{thm:shellability_2}]
 Inductively pierced codes have shellable polar complexes. 
 \end{theorem*}
 
\suspect{We use this shelling order again when we study Gr\"obner bases of toric ideals of inductively pierced codes.}\revision{We use this shelling order again when we study Gr\"obner bases of toric ideals of inductively pierced codes.}

Itskov, Kunin, and Rosen showed that if $\C$ is a nondegenerate hyperplane code, the polar complex of $\C$ is shellable \cite{itskov2018hyperplane}. Thus, our result that polar complexes of inductively pierced codes are shellable motivated us to look for hyperplane realizations of inductively pierced codes. We find that, indeed, inductively pierced codes are hyperplane codes. 

\begin{theorem*} [Theorem \ref{thm:hyperplane}.] Inductively pierced codes are nondegenerate hyperplane codes.\end{theorem*}

\suspect{Finally, we explore the toric ideals of inductively pierced codes, expanding upon the results in \cite{gross2016neural} that  toric ideals of all inductively 1-pierced codes have quadratic bases and toric ideals of all inductively 1-pierced codes on at most three neurons have quadratic Gr\"obner bases.  }

\suspect{
If $\C$ is an inductively $k$-pierced code, its toric ideal has a quadratic Gr\"obner basis. }

%

The organization of this paper is as follows.  Section \ref{defs} reviews basic definitions related to neural codes. We define inductively pierced codes and state some short results.  Section \ref{not_algebra} gives additional geometric and combinatorial results about inductively pierced codes. In particular, we show that their simplicial complexes are vertex-decomposable (Theorem \ref{thm:vertex_decomp}),  that their polar complexes are shellable (Theorem \ref{thm:shellability_2}), and that they are nondegenerate hyperplane codes (Theorem \ref{thm:hyperplane}).  \suspect{Finally, Section \ref{algebra}  introduces toric ideals and Gr\"obner bases and shows that toric ideals of inductively pierced codes have quadratic Gr\"obner bases (Theorem \ref{thm:grob}). } \revision{Finally, Section \ref{algebra}  introduces toric ideals and Gr\"obner bases and discusses toric ideals of inductively pierced codes.} Section 5 ends with a discussion and some open questions. 

\section{Background \label{defs}}
\subsection{Neural Codes and Realizations}
A \emph{neural code} on $n$ neurons is a subset $\mathcal \C$ of the boolean lattice $\mathcal P([n])$. We refer to the elements of $\C$, which are subsets of $[n]$,  as \emph{codewords}. An arrangement of sets $\mathcal U = \{U_1, \ldots, U_n\}\subset X$ defines a code $\code(\mathcal U, X)$. When the choice of ambient space $X$ is clear, we will write $\code(\mathcal U) = \code(\mathcal U, X)$. The codeword $c\subset [n]$ is an element of $\code(\mathcal U)$ whenever $$A_c = \bigcap_{i\in c} U_i \setminus \bigcup_{j\notin c} U_j  \neq \varnothing.$$ Equivalently, we can define $\code(\mathcal U)$  by labelling each point of $x\in X$ with the set of all $i$ such that $x\in U_i$, and defining $\mathcal C(\mathcal U)$ to be the set of labels used. We say that $\mathcal U$ is a \emph{realization} of $\code(\mathcal U)$ and $A_c$ is the \emph{atom} of $c$. 

A neural code $\mathcal C$ is \emph{convex} if it has a realization $\mathcal U = \{U_1, \ldots, U_n\}$ such that each $U_i$ is a convex open subset of $\R^d$. The minimal value of $d$ for  which this is possible is  the \emph{minimal embedding dimension. } A realization $\mathcal U = U_1, \ldots, U_n$ is \emph{nondegenerate} if the code $\code(\mathcal U)$ is stable under small perturbations of the arrangement. More formally, a realization is nondegenerate there exists $\epsilon >0 $ such that if $d_{\mathrm{Hausdorff}} (U_i, V_i) < \epsilon$ for all $i \in [n]$, $\code(U_1, \ldots, U_n) = \code(V_1, \ldots, V_n)$. 

We can narrow the concept of a convex code and look at hyperplane codes. A code $\C$ is a \emph{hyperplane code} if there exist half spaces  $\mathcal H = H_1^+, \ldots, H_n^+$ and a convex set $X$ such that $H_1^+\cap X, \ldots, H_n^+\cap X$ is a realization of $\C$.  That is, $\C = \code(\mathcal H, X)$. We call $\C$ a \emph{nondegenerate hyperplane code} if this realization is nondegenerate.  For an example of a nondegenerate hyperplane code, see Figure \ref{fig:hyperplane}.

\begin{figure}
\includegraphics[width = 3.5 in] {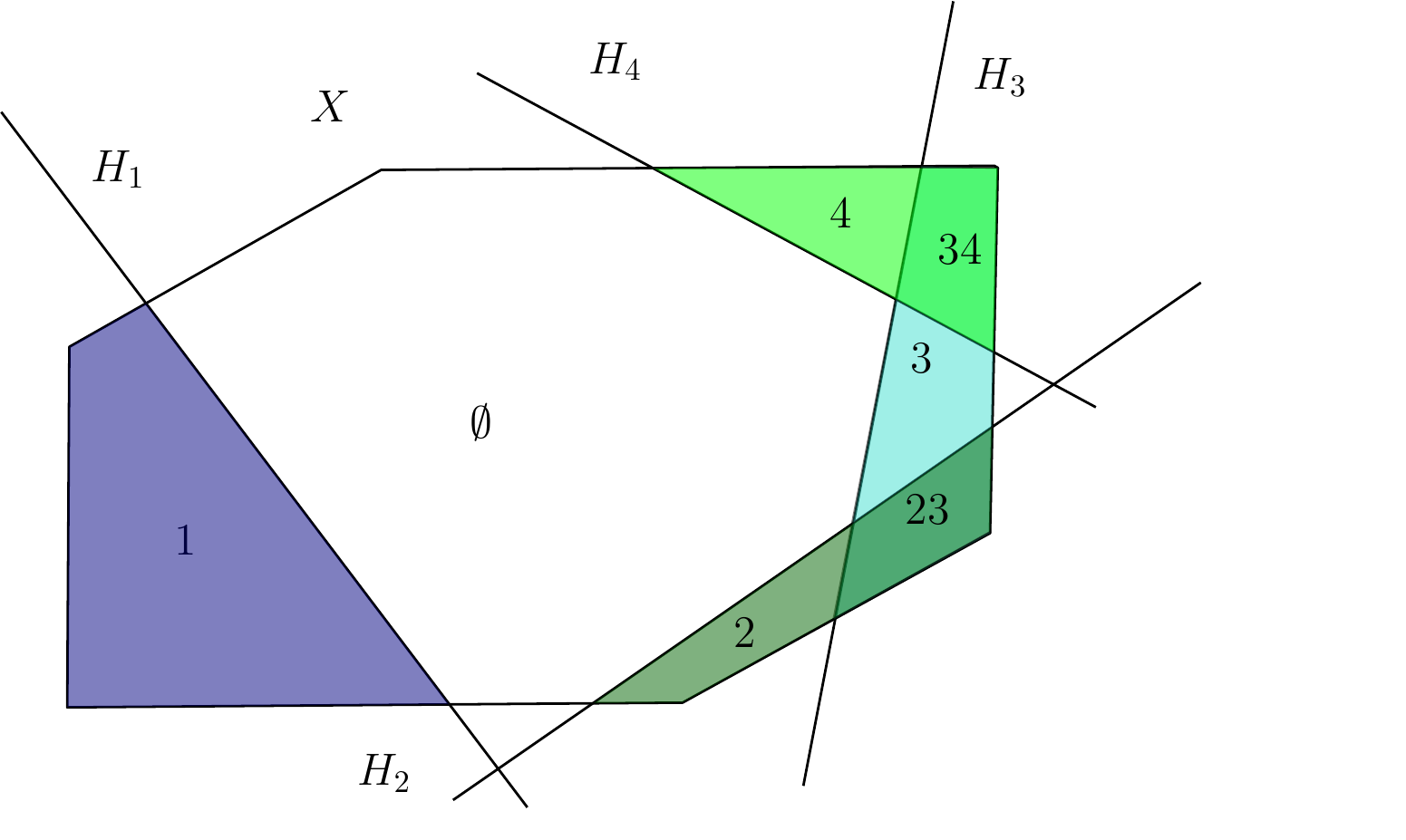}
\caption{A realization of a hyperplane code. Note the importance of the bounding convex set $X$: the half spaces $H_1^+, H_2^+$  and $H_4^+$ are disjoint, which is not possible if we take the bounding set to be all of $\R^2$. \label{fig:hyperplane}}
\end{figure}


\subsection{Inductively Pierced Codes}
We now define the operation of piercing in terms of neural codes. For more information, see \cite{gross2016neural}.  The input to the piercing operation is a neural code $\C$ on $n$ neurons and a partition of the set of neurons $[n]$ into three disjoint subsets $\lambda, \sigma$ and $\tau$. The output is a neural code $\pierce \C \lambda \sigma \tau $ on $n+1$ neurons, $\C\subset \pierce \C \lambda \sigma \tau$. This operation is not possible for all choices of $\C, \lambda, \sigma$ and $\tau$. 

\begin{definition}
Let $\C$ be a neural code on $[n]$ and $(\lambda,\sigma, \tau )$ be a partition of $[n]$, $|\lambda | = k$. 
We say that  $\C$ is $(\lambda, \sigma, \tau)$ pierceable if for all $\nu\subset \lambda$, $\sigma\cup \nu\in \C$. 
If $C$ is $(\lambda, \sigma, \tau)$ pierceable, the \emph{$k$-piercing} of $\lambda$ in $\C$ with respect to the background motif $\sigma, \tau$ is the code 
$$\pierce \C   \lambda  \sigma \tau = \C \cup \bigcup_{\nu\subset \lambda} \{\{\sigma\cup \nu \cup \{n+1\}\}\}$$
\end{definition}

\begin{definition}
A neural code $\C'$ on $n$ neurons is \emph{inductively $k$-pierced} if $n=1$ or if there exists a $k-$pierced code $\C$ on $n-1$ neurons such that $$\C' = \pierce \C \lambda \sigma \tau.$$
\end{definition}

\begin{figure}[!h]
\begin{center}
\includegraphics[width=4 in]{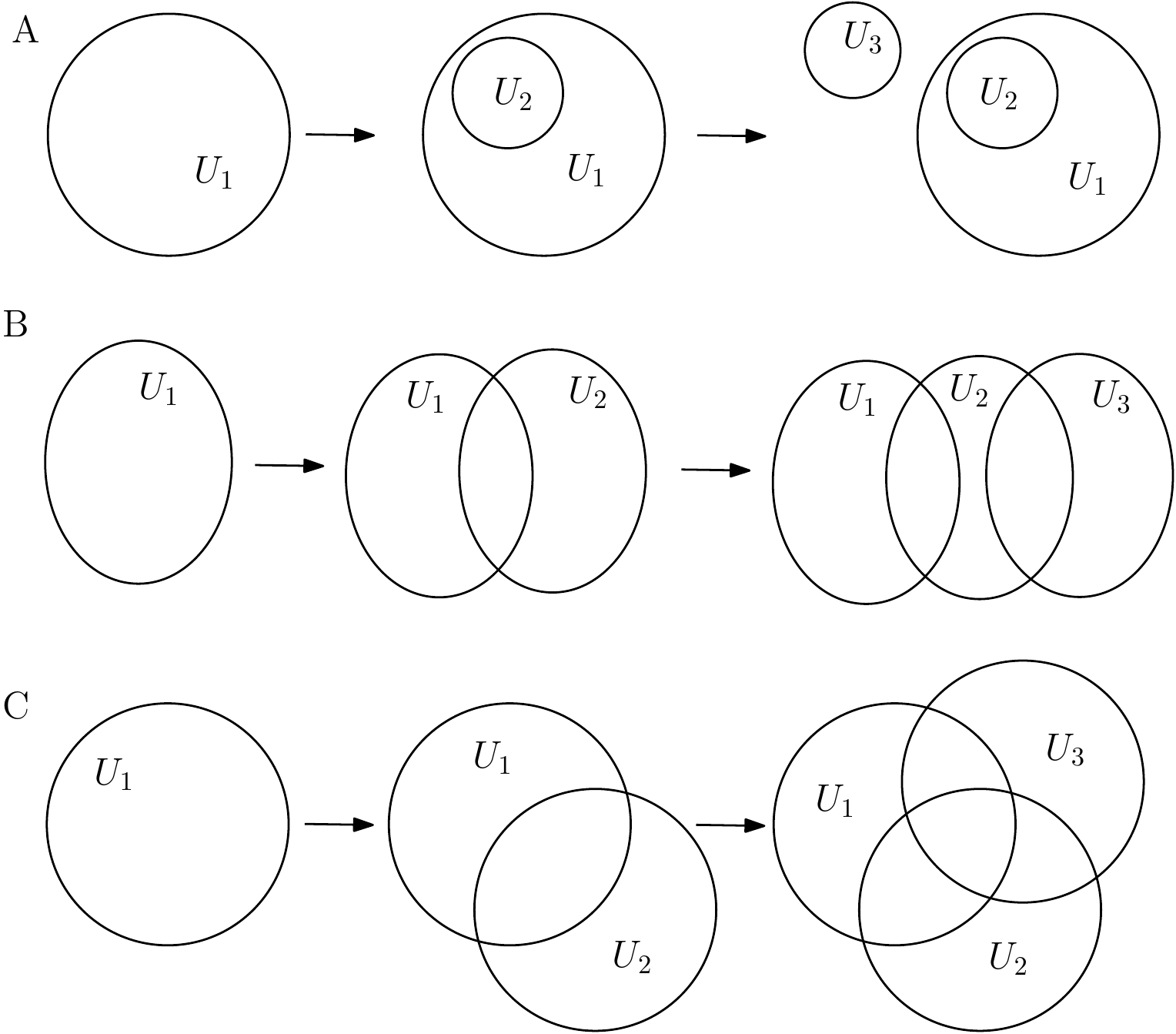}
\end{center}
\caption{Examples of 0, 1, and 2 piercings. (A) Constructing the code $\{\varnothing, 1, 12, 3\}$ in two steps by 0-piercing. (B) Constructing the code $\{\varnothing, 1, 12, 2, 13, 3\}$ in two steps by 1-piercing. See Figure \ref{fig:main_intro} for another example of an inductively 1-pierced code. (C) Constructing the  code $\{\varnothing, 1, 12, 2, 123, 13, 23, 3\}$ in two steps by 1-piercing, then 2-piercing. \label{fig:0_and_2}}
\end{figure}

For instance, the code $\{\varnothing, 1, 12, 2\}$ is inductively $1$ pierced, since we can construct it from the code $\{\varnothing, 1\}$ by 1-piercing the set  $\{1\}$ with respect to the background motif $(\varnothing, \varnothing)$. That is, 
$$\{\varnothing, 1, 12, 2\} = \pierce {\{\varnothing, 1\}} 1 \varnothing \varnothing$$
 The code $\{\varnothing, 1, 12, 2, 123, 13, 23, 3 \}$ is inductively 2-pierced, since we can construct it from $\{\varnothing, 1, 2, 12\}$ by 2-piercing the set $\{1, 2\}$ with respect to the background motif $(\varnothing, \varnothing)$. 
That is, 
$$\{\varnothing, 1, 12, 2, 123, 13, 23, 3 \} = \pierce {\{\varnothing, 1, 12, 2\}} {12} {\varnothing} {\varnothing}$$
 For a realization of this code, see Figure \ref{fig:0_and_2}. Notice that the choice of background motif as well as the set we pierce determines the code constructed:  
 $$\pierce {\{\varnothing, 1, 12, 2\} }  2 1 \varnothing = \{\varnothing, 1, 12, 2,  123, 13 \}$$
 however
 $$\pierce{\{\varnothing,  1, 12, 2\}}  2   \varnothing 1 = \{\varnothing, 1, 12, 2, 23, 3\}$$ 
 Also notice the requirement $\sigma\cup \nu \in \C$ for each $\nu\subset \lambda$. For instance, the code $\C=\{\varnothing, 1, 2\}$ is not 2-pierceable for any choice of $(\lambda, \sigma, \tau)$, since we would have to take $\lambda= \{1, 2\}$, $\sigma = \tau = \varnothing$, but $12\notin \C$. Likewise, we cannot 2-pierce the code $\C=\{\varnothing, 1, 12\}$, since this would require $\nu\in \C$ for all $\nu \subset \{1, 2\}$, but $2\notin \C$. 

Looking at realizations of $k$-pierced codes given in Figure \ref{fig:0_and_2}, we notice a few patterns. In an inductively $0$-pierced code, all receptive fields in the realization are either contained in one another or disjoint. In a code is  inductively 0-pierced if and only if, in any convex realization, boundaries of receptive fields do not intersect. (This is Proposition 2.7 of \cite{gross2016neural}.) In realizations of 1-pierced codes, the boundary of the receptive field most recently added intersects only with the boundary of one other receptive field. However,  this condition is not sufficient to characterize 1-pierced codes. In general, in a realization of a $k$-pierced code, the boundary of a receptive field added at the most recent step intersects with the boundaries of $k$ other receptive fields. Also notice that, in all realizations of inductively pierced codes given so far, the $U_i$ are disks. This is not a coincidence--in the next subsection, we show that this is always possible. 

\subsection{ Geometric realization}

Our pictures hint at the following result :
\begin{proposition} \label{prop:realization}An inductively $k$-pierced code has a convex realization by open balls of dimension $k+1$. In particular, the minimal embedding dimension of an inductively $k$-pierced code is at most $k+1$.\label{thm:real}
\end{proposition}

\begin{figure}[h!]\includegraphics[width = 3 in]{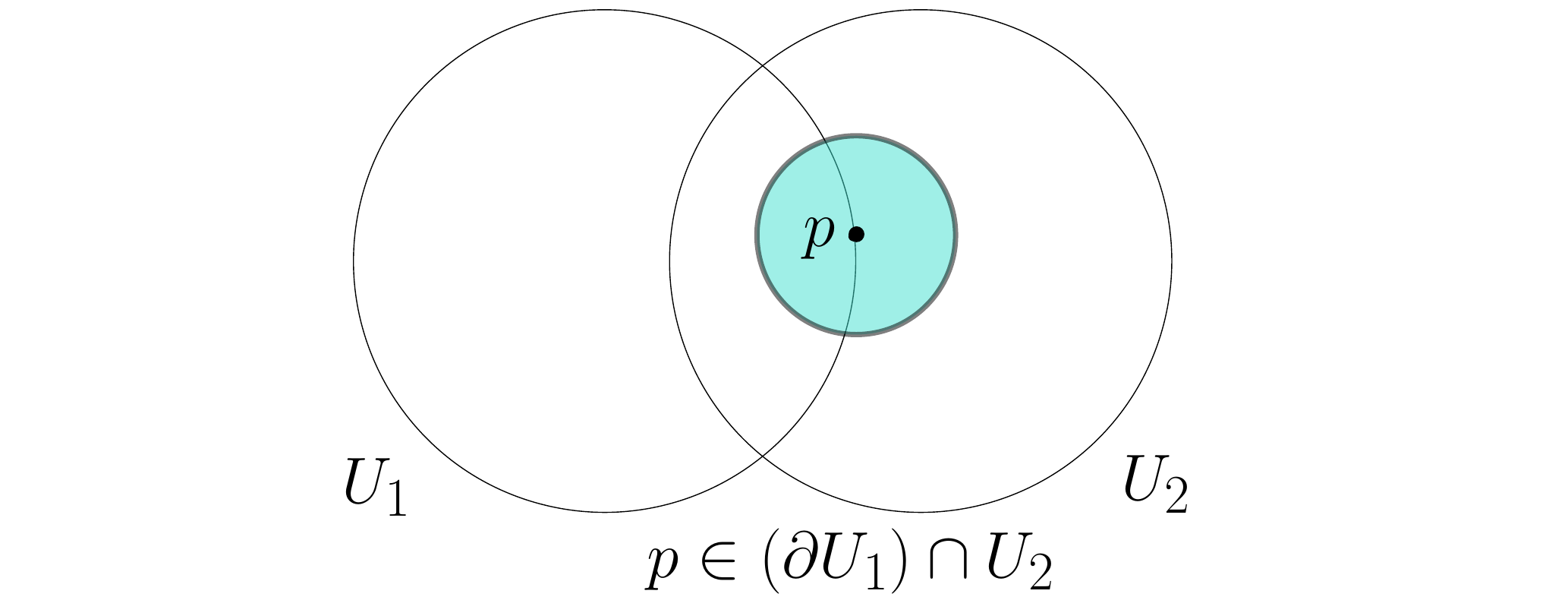}
\caption{To construct a realization of the code $\{\varnothing, 1, 12, 2, 123, 23\}$ from a realization of the code  $\{\varnothing, 1, 12, 2\}$,  we pick a point $p\in (\partial U_1) \cap U_2$ and let $U_3$ be a neighborhood of $p$ contained within $U_2$. }
\end{figure}

We give a detailed construction and proof in Appendix  \ref{app:balls}. Here, we give a proof sketch and a few examples of the construction. As the term ``inductively pierced" might suggest, we construct realizations of inductively pierced codes step by step, beginning with a realization of the one neuron code $\{\varnothing, 1\}$.  

\begin{figure}[h!]\includegraphics[width = 2 in]{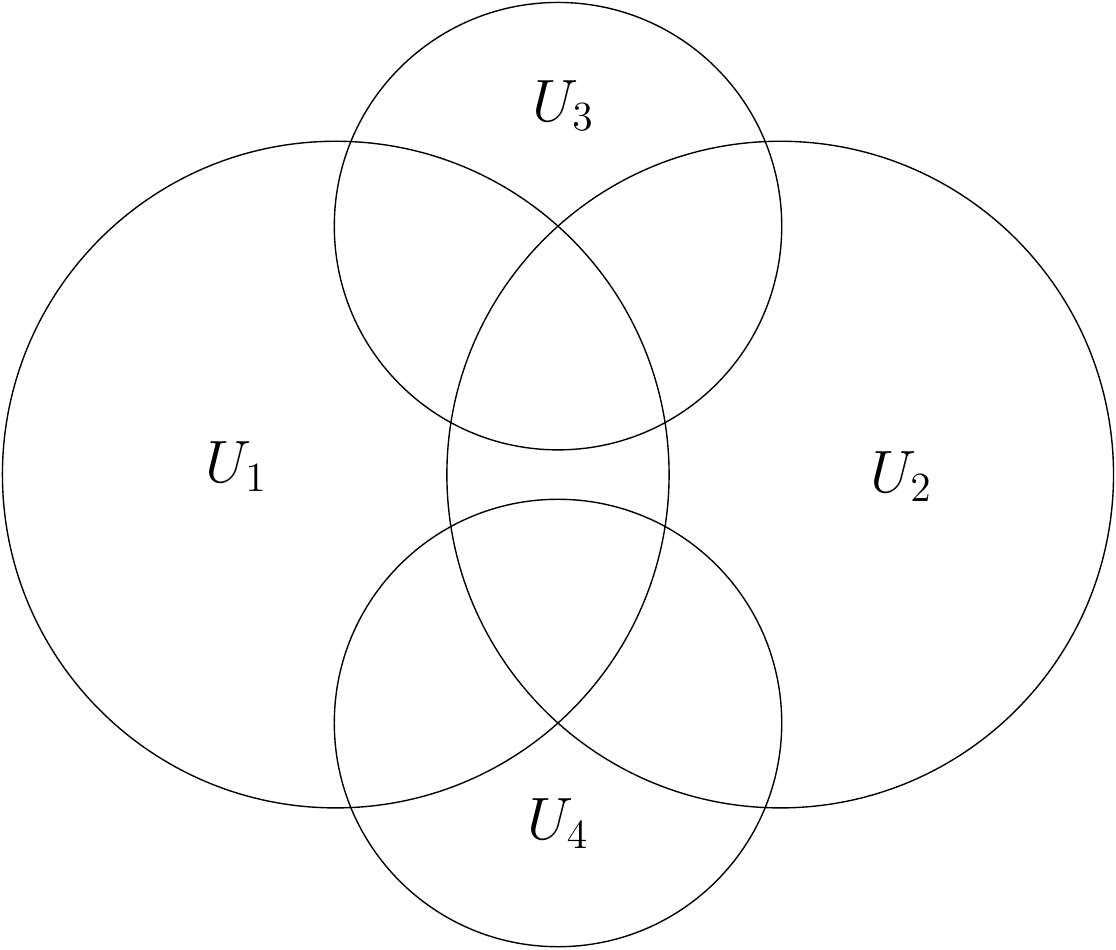}
\caption{An example of a code for which the dimension bound is sharp: $$\C' =  \{\varnothing, 1, 12, 2, 123, 13, 23, 3, 124, 14, 24, 4, 125, 15, 25, 5\}.$$ To see this, we first note that in any realization of the code $\C = \{\varnothing, 1, 12, 2, 123, 13, 23, 3, 124, 14, 24, 4\}$ with disks in $\R^2$, the sets $U_3$ and $U_4$ cover the intersection $\partial U_1\cap \partial U_2.$ However, to construct a realization of the code  $$\pierce \C {12} \varnothing {34} = \{\varnothing, 1, 12, 2, 123, 13, 23, 3, 124, 14, 24, 4, 125, 15, 25, 5\}$$ with disks in $\R^2$, we would need to place the disk $U_5$ over a point of $\partial U_1\cap \partial U_2$ not covered by $U_3$ or $U_4$. \label{fig:dim_bound_sharp} }
\end{figure}
 
Let $\C$ be an inductively $k$-pierced code  on $n-1$ neurons which is $(\lambda, \sigma, \tau)$ pierceable. To construct a realization of $\C' = \pierce \C  \lambda  \sigma\tau$, we begin with a realization of the code $\C$ with open sets $U_1, \ldots, U_{n -1} \subset \R^{k+1}$. Now, choose a point 
$$p\in \bigcap_{i\in \lambda} \partial U_i$$
 such that a small neighborhood around $p$ does not completely cover any atom in the realization of $\C$, and such that $p$ is interior to $\bigcap_{i\in \sigma} U_i\setminus \bigcup_{j\in \tau} U_j$. We take  $U_{n}$ to be a small ball around $p$. This gives a realization of $\C'$. To see this, we first check that any codeword of $\C$ is a codeword of $\code(U_1, \ldots, U_{n})$. This follows because we can choose the ball small enough not to cover up any atom $A_c$ for $c\in \C$. We next note that that since $p$ is on the boundary of $U_{i}$ for each $i\in \lambda$, $U_{n}$ meets each atom $A_{\sigma \cup \nu}$ for each $\nu\subset \sigma$, provided the boundaries are sufficiently nice. 

 In the detailed proof in Appendix \ref{app:balls}, we show that it is always possible to find such a point $p$ and a small open neighborhood $U_i$. This dimension bound is tight because, when we try to realize a $k$-pierced code in dimension $\ell < k+1$, it may not be possible to find an appropriate point $p$. For an example, see Figure \ref{fig:dim_bound_sharp}. 

\subsection{Unpacking the canonical form\label{subsec:unpacking}}
In this section, we define the neural ideal and interpret a previous result on neural ideals of inductively pierced codes. Recall that the Stanley-Reisner ideal of a simplicial complex is the ideal of  non-faces of simplicial complex: 
\begin{align*}
I_\Delta = \abk{\prod_{i\in \sigma} x_i | \sigma \in 2^{[n]}\setminus\Delta}\subset k[x_1, \ldots, x_n]
\end{align*}
The neural ideal extends this idea to arbitrary subsets of $\mathcal P([n])$. More specifically, the neural ideal $J_\C$ of a neural code $\C$ is defined as the ideal generated by indicator polynomials of ``non-codewords" of $\C$:
\begin{align*}
J_\C = \abk{\prod_{i\in \sigma} x_i \prod_{j\notin \sigma}(1-x_j) \mid  \sigma \in 2^{[n]}\setminus \C} \subset \mathbb F_2 [x_1, \ldots, x_n].
\end{align*}
We call a polynomial over $\mathbb F_2$ a $\emph{pseudo-monomial}$ if it is a product of terms of the form $x_i$, $(1-x_j)$.  The canonical form $CF(J_\C)$ of a neural ideal is a set of minimal (with respect to divisibility) pseudo-monomials which generate the ideal. That is
\begin{align*}
CF(J_\C) =  \{f\in J_\C | f \mbox{ is a minimal pseudo-monomial}\}.
\end{align*}
These correspond to minimal receptive field relationships. The elements of the canonical form come in three types, given in Table \ref{tab:CF}. 

\begin{table}[h!]
\begin{tabular}{@{}lll@{}}
\toprule
{Type} & Pseudo-monomial                                        & RF Relationship                                              \\ \midrule
Type 1                   & $\prod_{i\in \sigma} x_i $                          & $\bigcap_{i\in \sigma} U_i = \varnothing$                      \\ \midrule
Type 2                   & $\prod_{i\in \sigma}x_i \prod_{j\in \tau} (1-x_j) $ & $\bigcap_{i\in \sigma} U_i  \subseteq \bigcup_{j\in \tau} U_j$ \\ \midrule
Type 3                   & $\prod_{j\in \tau}(1- x_j) $                            & $X \subseteq \bigcup_{j\in \tau} U_j$                          \\ \bottomrule
\\
\end{tabular}
\caption{The three types of relationships in the canonical form $J_\C$. \label{tab:CF}}
\end{table}
 Theorem 3.6 of \cite{gross2016neural} states that if $\C$ is $k$-inductively pierced, then $$CF(J_\C) \subset \{f_{ij}| 1\leq i <j \leq n, f_{ij}\in \{x_ix_j, x_i(1-x_j), x_j(1-x_i)\}\}.$$ In other words, the canonical form of an inductively pierced code is quadratic. This means that all relationships between the $U_i$'s are determined by the pairwise relationships. This has several immediate corollaries. 

 First, inductively pierced codes are intersection complete. This follows from the characterization of canonical forms of inductively pierced codes, given above, and Proposition 3.7 of \cite{curto2018algebraic}, which  states that $\C$ is intersection complete if and only if each Type 2 pseudo monomial  $$\prod_{i\in \sigma} x_i \prod_{i\in \tau} (1-x_i)$$ of $CF(J_\C)$ has $|\tau|\leq 1$. 

Second, the simplicial complex $\Delta(\C)$ of an inductively pierced code is a clique complex. This follows because $\abk{CF^1(J_\C)}$, the ideal generated by the set of monomials in $J_\C$, is precisely the Stanley-Reisner ideal of $\Delta(\C)$. Now, the fact that $CF^1(J_\C)$ is quadratic implies that if $\sigma\subset [n]$ is a non-face of $\Delta(\C)$, there exist some $i, j\in \sigma$ such that $\{i, j\} \notin \Delta(\C)$.  Thus, if $\sigma$ is a clique in the 1-skeleton of $\Delta(\C)$, $\sigma\in \Delta(\C)$. 

%
%
%
%

\section{Simplicial and  polar complexes of inductively pierced codes \label{not_algebra}}
\subsection{Simplicial complexes, collapsibility, and vertex decomposability}
The simplicial complex $\Delta(\C)$ of a code is defined as the minimal simplicial complex containing $\C$. In other words, $\Delta(\C)$ is the simplicial complex consisting of all codewords of $\C$ together with all of their subsets. For instance, if $\C = \{1, 12, \varnothing\}$, $\Delta(\C) = \{1, 2, 12, \varnothing\}$. 

\begin{definition}[Links, deletions]
\
\begin{itemize}
\item  The link of a vertex is a simplicial complex is the subcomplex 
$$\lk_\Delta(v) = \{\sigma \setminus  v | v\in \sigma \in \Delta\}$$  
\item The deletion of a vertex in a simplicial complex is the subcomplex 
$$\mathrm{Del}_{\Delta}(v) = \{\sigma\setminus v| \sigma \in \Delta\}$$
\end{itemize}
\end{definition}

\begin{definition}[collapsible]
Let $\sigma$ be a maximal face (facet) of $\Delta$, $\tau\subsetneq \sigma$. Call $\tau$ a free face of $\Delta$ if it is not contained in any other facet. The complex $\Delta'$ is a \emph{collapse} of $\Delta$ if we can obtain it by removing all $\lambda$, $\tau\subseteq \lambda \subseteq \sigma$ for some free face $\tau$. A simplicial complex is \emph{collapsible} if there is a sequence of collapses which takes it to a point. \end{definition}
Note that every collapsible complex is contractible, but the converse is not true. 

Several definitions of vertex decomposable simplicial complexes exist in the literature. Here, we follow \cite{moradi2013vertex}. 

\begin{definition}[vertex decomposable]
A simplicial complex $\Delta$ is vertex-decomposable if either
\begin{itemize}
\item $\Delta$ is a simplex
\item There exists a vertex $v\in \Delta$ such that 
\begin{itemize}
\item $\lk_\Delta(v)$ is vertex decomposable.
\item $\mathrm{Del}_\Delta (v) = \{\sigma\setminus v | \sigma\in \Delta\}$ is vertex decomposable 
\end{itemize}
\end{itemize} 
\end{definition}

\begin{theorem}
If $\C$ is an inductively pierced code, $\Delta(\C)$ is a disjoint union of vertex decomposable complexes. \label{thm:vertex_decomp}
\end{theorem}

\begin{proof} We prove this statement via induction on $n$, the number of neurons.  As a base case, note that if $\C$ is a neural code on $1$ neuron, then $\C$ is inductively pierced and $\Delta(C)$ is a 0-simplex and is therefore vertex-decomposable. 

Now, suppose that the theorem holds for all inductively pierced codes on $n$ neurons. 
Let $\C'$ be an inductively pierced code on $n+1$ neurons. Then $\C'$ is a piercing of some inductively pierced code $\C$ on  $n$ neurons. That is, 
$$ \C' = \pierce \C \lambda \sigma \tau.$$ We first consider the case where $\lambda\cup\sigma$ is nonempty.  Then $$\lk_{\Delta(\C')}(n+1) = \Delta(\{\lambda\cup\sigma\})$$ and $$\mathrm{Del}_{\Delta(\C)} (n+1) = \Delta(\C),$$both  by the definition of the piercing operation. Now, by the inductive hypothesis, the connected component of $ \Delta(\C)$ containing $\lambda\cup\sigma$ is vertex decomposable. Therefore, by the definition of vertex decomposability, the connected component of $\Delta(\C')$ containing $\lambda\cup\sigma$ is vertex decomposable. Thus, in this case, $\Delta(\C')$ is the disjoint union of  vertex decomposable complexes. Finally, consider the case where $\lambda\cup \sigma = \emptyset.$ Then $\Delta(\C') = \Delta(\C) \cup \{n+1\}$, which is the disjoint union of vertex decomposable complexes. 
\end{proof}

Since any vertex-decomposable complex is collapsible, we obtain the corollary 

\begin{cor} If $\C$ is an inductively pierced code, $\Delta(\C)$ is the disjoint union of collapsible complexes. 
\end{cor}

In light of the results of Subsection \ref{subsec:unpacking}, the simplicial complex of an inductively pierced code is a vertex decomposable (hence, collapsible) clique complex. 
%

\subsection{The polar complex of an inductively pierced code is shellable\label{sec:polar}}
\begin{definition}[$\Gamma(\C)$]
The \emph{polar complex} of a neural code $\C$ is the simplicial complex $\Gamma(\C)$ with vertex set $\{1, 2, \ldots, n, \bar 1, \bar 2, \ldots, \bar n\}$  and facets $$\left \{\left(\bigcup_{i\in c} i\right) \cup \left(\bigcup_{j\notin c} \bar j\right)\mid c\in \C\right\}.$$
\end{definition}

For instance, the polar complex of  $\C = \{1, 12, \varnothing\}$ is $\Gamma(\C) = \Delta(\{1\bar 2, 12, \bar 1\bar2\})$. The polar complex of $\{\varnothing, 1,12, 2, 123, 13\}$ is the simplicial complex with facets $\{\bar 1\bar 2\bar 3, 1\bar 2 \bar 3, 12\bar 3, \bar 1 2 \bar 3, 123, 1\bar 2 3\}$. For a geometric realization of this complex, see Figure \ref{fig:polar_complex_ex}. Notice that when $\C$ is a code on $n$ neurons, $\Gamma(\C)$ is always a subset of the $n$-dimensional cross polytope. In particular, on two neurons $\Gamma(\C)$ is a subset of the square and on  three neurons $\Gamma(\C)$ is a subset of the octahedron.  Notice that while $\Delta(\C)$ discards information about the non maximal codewords of $\C$, $\Gamma(\C)$ retains the full data of $\C$. 

\begin{figure}
\includegraphics[width = 1.5 in] {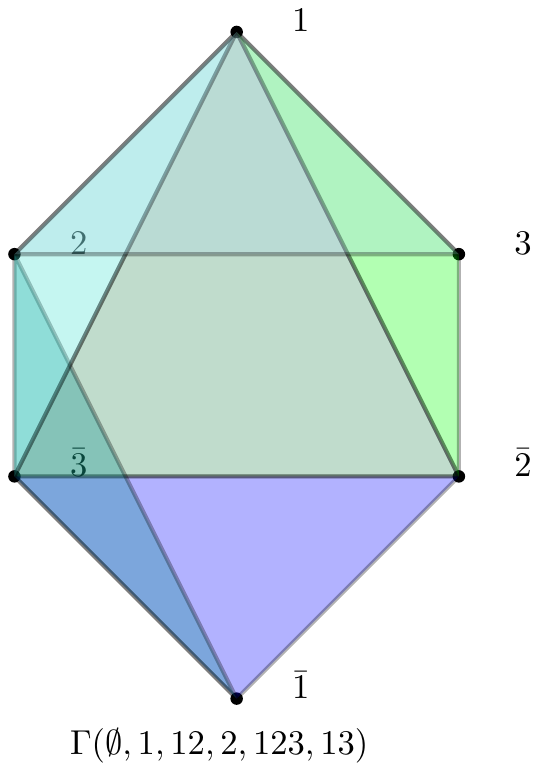}
\caption{The polar complex of $\{\varnothing, 1, 12, 2, 123,  13\}$. \label{fig:polar_complex_ex} }
\end{figure}

\begin{definition}
A pure $n$ dimensional simplicial complex is \emph{shellable} if there exists an ordering  $F_1, \ldots, F_m$ of its facets such that for each $1\leq k\leq m$, the complex $$\left(\bigcup_{1\leq i< k}{F_i}\right)\cap F_k$$ is pure and $n-1$ dimensional.
\end{definition}

\begin{theorem}\label{thm:shellability_2}
If $\C$ is an inductively pierced code,  $\Gamma(\C)$ is shellable. 
\end{theorem}
Before proving this theorem, we illustrate it with an example. Let $\C = \{\varnothing, 1, 12, 2, 123, 23 \}$, note that $\C$ is inductively pierced and $\Gamma(\C)$ has facets $$\{ \bar 1\bar 2\bar 3,  1\bar2\bar 3, 12\bar3, \bar12\bar3, 123, \bar123\}.$$ 
We shell $\Gamma(\C)$ by first shelling $\Gamma(\C)$ down to the cone over $$\Gamma(\{\varnothing, 1, 12, 2\}) = \{\bar 1\bar 2, 1\bar 2, 12, \bar 12\}$$
and then shell this cone by induction. More precisely: notice that $$\lk(3) = \{\bar12, 12\},$$ which is the cone over the 0-dimensional cross polytope $\{1, \bar 1\}$. Use a shelling order for the 0-dimensional cross polytope to induce an order on the facets of $\Gamma(\C')$ which include $3$. That is, pick $F_5 = 123,  F_ 6 = \bar123$. 
Notice that the intersection of $F_6$ with $ F_1\cup  \cdots \cup F_5$ is $\{23, \bar12\}$ and the intersection of $F_5$ with $F_1\cup \cdots\cup F_4$ is $12$, both of which are pure and 1 dimensional. So, removing these two facets last is valid. Once we do this, we have shelled down to the cone over $\Gamma(\C)$. By induction, we have the shelling order  $F_ 1 = \bar 1\bar 2, F_2 = 1\bar 2, F_3 = 1 2, F_4 = \bar 12$ for $\Gamma(\C)$. Putting these pieces together gives us the shelling order $F_ 1 = \bar 1\bar 2\bar 3, F_2 = 1\bar 2\bar 3, F_3 =1 2 \bar 3, F_4 =\bar 1 2 \bar 3, F_5 = 123, F_ 6 =\bar 123$ for $\Gamma(\C')$. 

Now, we begin the proof  that the polar complex of an inductively pierced code is shellable by defining the shelling order $<$ on the facets of $\Gamma(\C)$.  Since facets of $\Gamma(\C)$ correspond to codewords of $\C$, $<$ is an order on the codewords of $\C$. We will use this order again in Section \ref{algebra} to define the monomial order used in \suspect{Theorem} \revision{Conjecture} \ref{thm:grob}.

Note that $\max(c)$ gives the highest index of a neuron in $c$. If $\C$ is inductively pierced, this corresponds to the piercing step at which $\C$ was added. We define the relation $c<d$ as follows: 

\begin{itemize}
\item The codeword added first comes first: $c<d$ if $\max(c) < \max(d)$
\item  If two codewords are added at the same piercing step, we break the tie by putting the codeword of higher weight first: $c<d$ if $\max(c) = \max(d)$ and $|c|>|d|$
\item   If two codewords are added at the same piercing step and have the same weight, we break the tie using the lexicographic order.: $c<d$ if  $\max(c) = \max(d)$, $|c|=|d|$,  and $c<d$ lexicographically.
\end{itemize}

For instance, $<$ induces the order  $1< 12< 2< 123 < 23$ on $\C = \{\varnothing, 1, 12, 2, 123, 23\}$. The requirement that ties between codewords of the same weight added at the same step are broken lexicographically is not used in any of our proofs; any way of breaking this tie will work. 

We state a few definitions and results pertaining to shellability which we will use in these proofs. We will use the following characterization of a shelling order: $F_1, \ldots, F_m$ is a shelling order if and only if for each $\sigma \in F_i\cap F_j$, $i< j$, if $\dim \sigma < \dim F_j-1$, there exists $\ell < j$, $\tau = F_\ell\cap F_j$, $\sigma \subset \tau$, $\dim \tau = \dim F_j-1$ \cite{kozlov2007combinatorial}.  

We will also use the notion of a simplicial join and the elementary fact that the simplicial join of shellable complexes is shellable \cite{bjorner1997shellable}. 

\begin{definition}[Simplicial Join]
Let $\Delta_1$, $\Delta_2$ be abstract simplicial complexes. The simplicial join of $\Delta_1$ and $\Delta_2$ is the simplicial complex $\Delta_1*\Delta_2 =  \{\sigma\cup \tau \mid \sigma\in \Delta_1, \tau\in \Delta_2\}$. 
\end{definition}

\begin{lemma}\label{lem:weight_order}Let the weight of a facet of $\Gamma(\C)$ be the number on-vertices it contains. Any ordering of the facets of the $n$ dimensional cross polytope $\Gamma(2^{[n]})$  by weight is a shelling order. 
\end{lemma}

\begin{proof}[Proof of Lemma]

Without loss of generality, let $F_1, \ldots, F_{2^n}$ be ordering of the facets which is nondecreasing in weight. Now, let $\sigma\in F_i\cap F_j$, $1\leq i\leq j \leq 2^n$. Let $F_j$ be of weight $k$. If  $\dim \sigma < n-1$, we show that there exists $1\leq \ell \leq j$, $\tau \in F_\ell \cup F_j$ such that $\sigma \subset \tau$ and $\dim \tau  = n-1$. Since our ordering is nondecreasing in weight and $i\leq j$, there is some on-vertex $m$ in $F_j$ whose corresponding off-vertex $\bar m$ is contained in $F_i$. Then let $F_\ell$ be $F_j\setminus m \cup  \bar m$. We have $\ell < j$ because $F_\ell$ is lower weight than $F_j$ by construction. Further, $\tau = F_\ell\cap F_j$ is $n-1$ dimensional by construction. Finally, $\sigma \subset  \tau$, since by construction, all vertices of $\sigma$ are vertices of $F_j$, and neither $m$ nor $\bar m$ is a vertex of $\sigma$, and $\tau$ contains all vertices of $F_j$ other than $m$. Thus $\tau$ is pure $n-1$ dimensional. Thus this is a shelling order. To obtain this result for non increasing orders, flip the labels on the on- and off-vertices.  \end{proof}

Now, we are ready to prove Theorem \ref{thm:shellability_2}, that $<$ gives a shelling order on $\Gamma(\C)$. 


\begin{proof}
We proceed by induction on the number of neurons. As a base case, note that $\bar 1 < 1$ is a shelling order for $\Gamma(\{\varnothing, 1\}) = \{1, \bar 1\}$. 

Now, let $\C'$ be an inductively pierced code on $n$ neurons. Then $\C' =\pierce \C \lambda \sigma \tau$ for some inductively pierced code $\C$ on $n-1$ neurons.  We give a procedure for shelling $\Gamma(\C')$ down to a cone over $\Gamma(\C)$. Let $\Sigma$ be the facet of $\Gamma(\C)$ specified by $(\sigma, \tau)$. That is, $\Sigma = \{i: i\in \sigma\} \cup \{ \bar j : j\in \tau\}$. 

Notice that $\lk(n)$ is the simplicial join of the full $|\lambda|$-dimensional cross polytope on the vertex set $\{i, \bar i |i\in \lambda\}$ and the face $\Sigma$.   The open star of $n$ is the cone of $\lk(n)$ with the point $n$. By Lemma \ref{lem:weight_order} and the fact that $<$ orders codewords by weight, $<$ induces a shelling order $G_1, \ldots, G_{2^k}$ on the open star of $n$. 

Now, since $\C$ is an inductively pierced code on $n-1$ neurons, by the inductive hypothesis, $<$ gives a shelling order on the facets of $\Gamma(\C)$. Since the closed star of $\bar n$ in $\Gamma(\C')$ is the cone over $\lk(\bar n) = \Gamma(\C)$, the closed star of $\bar n$ in $\Gamma(\C')$ is shellable. Let this shelling order be $F_1, \ldots, F_{n-2^{k}}$. We claim that  $F_1, \ldots, F_{n-2^{k}}, G_1, \ldots, G_{2^k}$ is a shelling order for $\Gamma(\C')$. Further, note that this is the order induced by $<$. 

We must check that 
\begin{align}
\left(F_1\cup \cdots\cup F_{i-1}\right)  \cap  F_i \mbox{ is pure and $n-1$ dimensional} \\
\left(F_1\cup  \cdots \cup F_{n-2^{k}}\cup G_1\cup\cdots\cup G_{i-1}\right) \cap  G_i \mbox{ is pure and $n-1$ dimensional} 
\end{align} 
for all appropriate values of $i$. Statement (1) follows from the fact that $F_1, \ldots, F_{n-2^{k}}$ is a shelling order. 

Now, we check statement 2. Notice 
\begin{align*}
\left(F_1\cup \cdots\cup F_{n-2^{k}}\cup G_1\cup \cdots\cup G_{i-1} \right) \cap  G_i  =\\  \left(G_i \cap \left(F_1\cup  \cdots \cup F_{n-2^{k}}\right)\right) \cup \left(G_i  \cap \left(G_1\cup \ldots\cup G_{i-1}\right)\right).
\end{align*}
To check that this is pure and $n-1$ dimensional, we need to prove that each of these intersections is pure and $n-1$ dimensional.  First, we have that  $G_i \cap (G_1\cup \cdots\cup G_{i-1})$ is pure and $n-1$ dimensional because  $G_1, \ldots, G_{2^k}$ is a shelling order for the closed star of $n$. Next, from the definition of piercing, we have that $(F_1\cup\cdots\cup F_{n-2k}) \cup G_i  =( \bar n* \Gamma(\C))\cap G_i= G_i \setminus n$, which is pure and $n-1$ dimensional. Thus, $F_1, \ldots, F_{n-2^{k}}, G_1, \ldots, G_{2^k}$ is a shelling order for $\Gamma(\C')$. 

\end{proof}

\subsection{Inductively pierced codes are hyperplane codes} 
Itskov, Kunin, and Rosen show that if $\C$ is a nondegenerate hyperplane code, $\Gamma(\C)$ is shellable \cite{itskov2018hyperplane}. Further, they showed all other known combinatorial characterizations of nondegenerate hyperplane codes follow from the shellability of $\Gamma(\C)$. 

Since we have shown that inductively pierced codes have shellable polar complexes, it is natural to ask whether inductively pierced codes are nondegenerate hyperplane codes. The answer to this question is yes:

\begin{figure}[!h]
\includegraphics[width = 5  in]{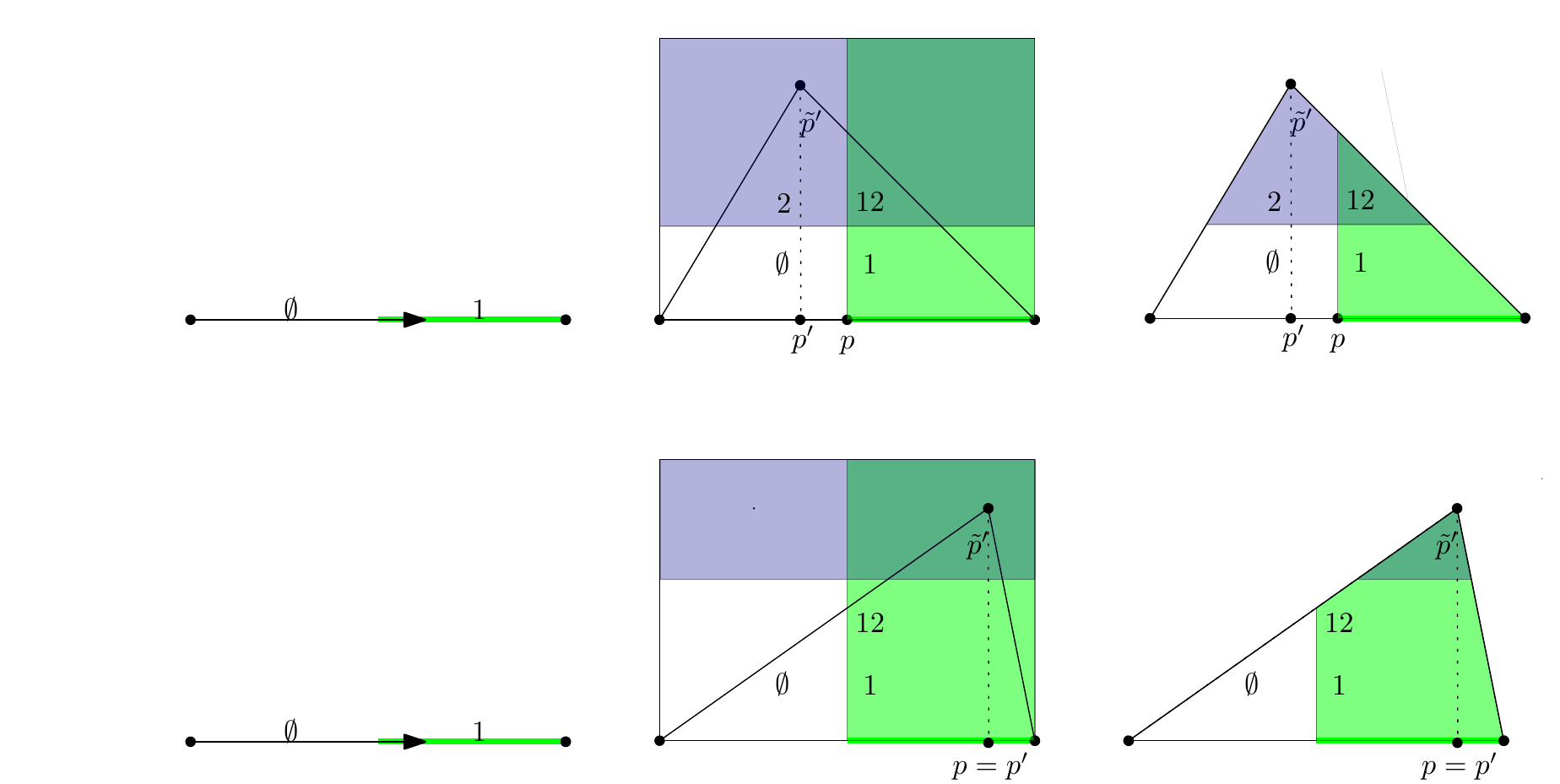}
\caption{Using our construction to give hyperplane realizations of the inductively pierced  codes $\{\varnothing, 1, 12, 2\}$ and $\{\varnothing, 1, 12\}$. \label{fig:hyper_const}} 
\end{figure} 

\begin{theorem} Inductively pierced codes are nondegenerate hyperplane codes. In particular, if $\C$ is an inductively pierced code on $n$ neurons, we can construct an $n$ dimensional nondegenerate hyperplane realization $(\mathcal H, X)$ of $\C$ such that $X$ is an $n$-simplex.   \label{thm:hyperplane}
\end{theorem}
We give a proof sketch and here and reserve the full proof for  Appendix \ref{app:hyperplanes}. Our construction of hyperplane realizations builds off of the construction for sphere realizations in the previous section. 
As before, we construct these realizations inductively. We start with a hyperplane realization of an inductively pierced code $\C$ on $n$ neurons and modify it construct a realization of  the code $\C' = \pierce{\C} {\lambda} {\sigma} {\tau}$ on $n+1$ neurons. As before, pick a point $$p\in \bigcap_{i\in \lambda} \partial U_i$$ such that a small neighborhood around $p$ does not cover up any atom in the realization of $\C$, and such that $p$ is interior to $\bigcap_{i\in \sigma} U_i\setminus \bigcup_{j\in \tau} U_j$. With the new condition that the $U_i's$ are half-spaces, it is no longer possible to choose $U_n$ to be a small open ball around $p$. To get around this, we lift the point $p$ into an extra dimension in the ambient space, and then use a hyperplane to slice off a small neighborhood of $p$ in the ambient space. More explicitly, we consider our realization as living in an $n$-dimensional subspace of $\R^{n+1}$, and construct a $n+1$ dimensional realization of $\C$ by extending the hyperplanes and the bounding set $X$ perpendicular to this subspace. We then choose a point $p'$ above $p$ and take the bounding convex set $X'$ of the new realization to be the cone $p'*X$. We then choose the hyperplane $H_{n+1}$ which slices a small neighborhood of $p'$ off  the top of the cone. Since this neighborhood projects down to a small neighborhood around $p$, this gives a realization of $\C'$. 
In order to produce a \emph{nondegenerate} hyperplane realization, we can modify this construction slightly by replacing $p$ with a nearby point $p'$ which is not contained on any of the hyperplanes.

We demonstrate this construction on codes $\{\varnothing, 1, 12, 2\}$ and $\{\varnothing, 1, 12\}$. See Figure \ref{fig:hyper_const} for an illustration.  We begin with the realization of the code $\{\varnothing, 1\}$. To construct a realization of $\{\varnothing, 1, 12, 2\}$, we pick $p$ on $H_1$, and pick $p'$ near $p$. We extend $H_1$ vertically upwards, and pick a point $\tilde p'$ above $p'$.  We choose the cone with point $\tilde p'$ as the bounding polytope of the new realization. We choose the half space $H_2^+$ cut off by a hyperplane parallel to the hyperplane containing the realization of $\{0, 1\}$, oriented up. To construct $\{\varnothing, 1, 12\}$, we proceed in essentially the same way, but choose $\tilde p'$ to be in the interior of the atom $A_1$. We choose our half space $H_2^+$ such that the neighborhood it cuts off is completely contained within $A_1$.

Note that Theorem \ref{thm:shellability_2} arises again as a corollary of Theorem \ref{thm:hyperplane}. 

\begin{cor} \label{cor:shelling}
If $\C$ is an inductively pierced code, $\Gamma(\C)$ is shellable. 
\end{cor}

This shelling order arises from a sweeping hyperplane argument, similar in spirit to \cite{chen2018neural}. In short, we sweep a hyperplane across $\R^n$ and record the order in which it encounters each atom. This induces an ordering on the codewords of $\C$, which in turn induces an ordering on the facets of $\Gamma(\C)$. 

We can therefore exploit this construction, together with the fact that the bounding convex set for our hyperplane realization of $\C$ is a simplex, to construct a shelling order of $\Gamma(\C)$ which respects the piercing order.  More precisely, we can find a shelling order such that facets whose codewords are added at the $i$-{th} piercing step come before facets whose codewords are added at the $(i+1)$-th step. 

To do this, we first notice that, since the bounding convex set of our arrangement is a simplex, its vertices are in general position. Thus, we can choose a sweep direction such that a sweeping hyperplane meets them in any order we choose. Choose a sweeping hyperplane which meets the vertices in the order they were added during the inductive construction. Now, assuming that we take our $H_i$ to pass sufficiently close to $p_i$, this sweeping hyperplane encounters codewords in the intended order. 

However, it is not clear that we can choose a sweep direction which produces the shelling order $<$, since it would need to encounter the set of codewords added at the same step in order of decreasing weight.  We leave this open:
\begin{question}
Given an inductively pierced code $\C$, can we always choose a hyperplane realization of $\C$ and a sweep direction $w$ such that the shelling order $<$ arises from sweeping a hyperplane across a hyperplane realization of $\C$?
\end{question}

\section{Toric ideals of inductively pierced codes\label{algebra}}

\subsection{Neural toric ideals}
In this section, we define the toric ideal of a neural code \suspect{and prove that the toric ideal of an inductively k-pierced code has a quadratic Gr\"obner basis. On the way to doing this, we prove a lemma that may be interesting in its own right, that the toric ideal of a subset of a neural code is contained within the toric ideal of the original code.}\revision{In this section, we define the toric ideal of a neural code and summarize existing theorems and conjectures about toric ideals of inductively pierced codes. We prove that the toric ideal of a subset of a neural code is contained within the toric ideal of the original code.} In this respect, the neural toric ideal is better behaved than the neural ideal. Our definition of the toric ideal of a neural code comes from \cite{gross2016neural}. For more background on toric ideals, see Chapter 4 of \cite{sturmfels1996grobner}.

\begin{definition}[Neural toric ideal]
Let $\C$ be a neural code on $n$ neurons and $\C^*=\C\setminus \varnothing$. Define the\emph{ codeword ring} of $\C$ to be $W_\C=k[\{y_c|c\in \C^*\}]$ and the \emph{neuron ring} to be $N_\C=k[x_1, \ldots, x_n]$. We will call the $y_c$ codeword variables and the $x_i$ neuron variables.  Let a homomorphism $\phi_\C:N_\C\to W_\C$ be given by 
\[y_c\mapsto \prod_{i\in c}x_i.\]
Then the \emph{toric ideal} $T_\C$ is the kernel of $\phi_\C$.  
\end{definition}

\begin{example} Let $\C = \{\varnothing,1, 2, 12\}$. Then $\phi_\code(y_{12}) = x_1x_2=\phi_\code(y_1y_2)$, so $y_1y_2-y_{12}\in T_\C$. In fact, $T_\C = \abk{y_1y_2-y_{12}}$. 
\end{example}

For those familiar with toric ideals, this relates to the ordinary definition of a toric ideal as follows: Construct a matrix  $M_\C$ whose columns are the nonempty codewords of $\C$, written in vector notation. The toric ideal $T_\C$ is then the toric ideal of the matrix $M_\C$, using the standard notion of the toric ideal of an integer matrix. Since neural codes are defined by the same combinatorial data as hypergraphs, our definition of the toric ideal of a neural code also lines up with the definition of the toric ideal of a hypergraph \cite{petrovic2017hypergraph, petrovic2014toric, gross2013combinatorial}. \subsection{Gr\"obner bases}

Gr\"obner bases depend on monomial orders. 
Define $x^a=x_1^{a_1}\cdots x_n^{a_n}$ for $a\in \mathbb N^n$.  Over $k[x_1, \ldots, x_n]$, we say that $\prec$ is a valid monomial order if the following hold: 
\begin{enumerate}
\item It is multiplicative: that is, $x^a\prec x^b$ implies $x^{a+c}\prec x^{b+c}$ for all $a, b, c\in \N^n$. 
\item The constant monomial is the smallest: that is, $1\prec x^a$ for all $a\in \N^n$. 
\end{enumerate}
For instance, the lexicographic order on $k[x_1, \ldots, x_n]$ gives that $p\prec q$ if the exponent of $x_1$ is greater in $p$ than in $q$. If this creates a tie, we break it by comparing the exponent of $x_2$, and so on and so forth. 
We denote the \emph{leading term} of the polynomial $f$ according to the monomial order $\prec$ by  $LT_\prec(f)$. The \emph{initial ideal} $in_\prec(I)$  of an ideal $I$ is the ideal generated by the leading terms of all polynomials in $I$. 

\begin{definition}[Gr\"obner basis]
A subset $\mathcal G$ is a Gr\"obner basis for $I$ if the leading term of each member of $I$ is divided by the leading term of a member of $\mathcal G$. That is, $\mathcal G$ is a Gr\"obner basis if 
\[in_\prec(I)=\abk{LT_\prec(g):g\in  \mathcal G}.\] 
\end{definition}

\subsection{Toric ideals of inductively pierced codes}
The authors of \cite{gross2016neural} give several criteria for identifying 0- and 1-inductively pierced codes using $T_\C$. They give a necessary and sufficient condition for identifying inductively 0-pierced codes: under some light assumptions on $\C$, $T_\C = \abk{0}$ if and only if $\C$ is inductively 1-pierced (Theorem 4.1). Next, they give a necessary condition for identifying inductively 1-pierced codes: if $\C$ is inductively $1$-pierced, then the toric ideal $T_\C$ is generated by quadratics or $T_\C = \abk{0}$ (Theorem 4.5). 

However, this condition is not sufficient. For instance, the neural code\\ $\C = \{\emptyset, 1, 2, 3, 12, 13, 23, 123\}$ is not inductively 1-pierced, but its toric ideal is generated by quadratics: 
\begin{align*}
T_C =  \abk{y_{110} - y_{100}y_{010}, y_{101} - y_{100}y_{001}, y_{011} - y_{010}y_{001}, y_{111} - y_{110}y_{001}}.
\end{align*}
Thus, the authors of \cite{gross2016neural} instead look for a necessary and sufficient condition for a code to be inductively 0- or 1-pierced in the Gr\"obner basis of the toric ideal with respect to a particular term order. For $n=3$, they find such a term order: assuming again some light assumptions on $\C$, they found that a code on 3 neurons is 1-inductively pierced if and only if the Gr\"obner basis of $T_\C$ with respect to the weighted graded reverse lexicographic order with the weight vector $w = (0,0,0,1,1,1,0)$ contains only binomials of degree 2 or less (Proposition 4.8).

Based on these results, the authors of \cite{gross2016neural} made the following conjecture:

\begin{conjecture*}[\cite{gross2016neural}, Conjecture 4.9]
For each $n$, there exists a term order such that a code is  0- or 1-inductively pierced if and only if the reduced Gr\"obner basis contains binomials of degree 2 or less. 
\end{conjecture*}

\suspect{We view our Theorem \ref{thm:grob} as a stronger form of Theorem 4.5 and a partial answer to Conjecture 4.9. Namely, we show that if a code is inductively 1-pierced, there is some term order such that the reduced Gr\"obner basis contains binomials of degree 2 or less. However, while our term order does not depend on the particular inductively pierced code, it does depend on the neurons in the inductively pierced code being labeled such that the $i^{th}$ neuron is added as a piercing at the $i^{th}$ step. Further, in this term order, the toric ideals of many codes which are not 0- or 1- inductively pierced have quadratic Gr\"obner bases. In particular, any $k$-inductively pierced code has a quadratic Gr\"obner basis in this term order. }

\revision{}

\suspect{In order to prove this theorem, we first show that toric ideals behave well under inductive constructions:}

\revision{ Toric ideals behave well under inductive constructions on codes:}

\begin{lemma}
 If $\C \subset \C'$, then $T_\C\subseteq T_{\C'}$.  \label{lemma:nested}
\end{lemma}

\begin{proof}
Let $\C$ be a neural code on $[n]$, $\C'$ be a neural code on $[n']$, $n\leq n'$, $\C\subseteq \C'$. Let $W_\C = k[x_1, \ldots, x_n]$, $W_{\C'} = k[x_1, \ldots, x_{n'}]$, $N_\C = k[y_c |c\in \C]$ $N_{\C'} = k[y_c |c\in \C']$. 

Define the maps  $w: W_{\C}\to W_{\C'}$ and $n:N_{\C}\to N_{\C'}$ to be inclusion maps. 
We show that the following diagram commutes:

\begin{center}
\begin{tikzcd}
W_\C \arrow[r, "w"] \arrow[d, "\phi_\C"]
& W_{\C'} \arrow[d, "\phi_{\C'}"] \\
N_\C \arrow[r, "n" ]
& N_{\C'}
\end{tikzcd}
\end{center}

Recall that $\phi_\C, \phi_{\C'}$ are defined by \[\phi_{\C}(y_c)=\prod_{i\in c} x_i\, , \qquad \phi_{\C'}(y_c')=\prod_{i\in c'}x_i.\] Thus, 
\[n\circ \phi_\code(y_c)=n\left(\prod_{i\in c} x_i\right)=\prod_{i\in c} x_i\]
and 
\[\phi_{\C'}\circ w (y_{c})=\phi_{\C'}\circ y_c =\prod_{i\in c'} x_i=\prod_{i\in c} x_i,\]
since $\phi_\C$ and $\phi_{\C'}$ give the same map when restricted to $\C$. Thus, the diagram commutes and $T_\C = w(T_\C)\subset T_{\C'}$.  

\end{proof}

\begin{cor} If $\C'$ is a piercing of $\C$, $T_\C \subseteq T_{\C'}$. 
\end{cor}

In Section \ref{sec:polar}, we defined an order $<$ on the codewords of a neural code $\C$, and showed that it defined a shelling order on $\Gamma(\C)$. We now reuse this order to define a monomial order on  $k[y_c|c\in \C^*]$. We let $\prec$ be the lexicographic term order for $k[y_c|c\in \C^*]$ with the $y_c$ ordered by $y_c\prec y_d$ if $c<d$.

For instance, $y_{12} \prec y_1y_2$, since $ \{1,2\}< \{2\} $. 

Note that our definition of $<$ is completely independent of the particular neural code we are working with. In particular, $<$ gives a total ordering on the set of finite subsets of the $\mathbb Z_{> 0}$. Therefore $\prec$ gives a monomial order on $k[y_c |c\in 2^{[n]}\setminus \{\varnothing\}]$ for all $n$. 
 Therefore, if $p, q\subset k[y_c|c\in \C^*]\subset k[y_c|c\in \C'^*] $, $p\prec q$ when we are working in $k[y_c|c\in \C^*]$ if and only if  $p\prec q$ when we are working in $k[y_c|c\in \C'^*]$. Along with Lemma \ref{lemma:nested}, this allows us to use induction while working with toric ideals of inductively pierced codes.

\begin{conjecture}
\revision{If a neural code is labeled such that the $i^{th}$ neuron is added as a piercing at the $i^{th}$ step, then its toric ideal has a quadratic Gr\"obner basis with respect to the term order $\prec$.}
\end{conjecture}

\revision{Code used to generate computational evidence of this conjecture can be found at \url{https://github.com/lienkaemper/Neural-Toric-Ideals/blob/master/Downloads/neural_toric_examples.ipynb}}

\revision{The existence of a quadratic Gr\"obner basis for $T_\C$ does \emph{not} follow from the other combinatorial properties of inductively pierced codes discussed in this paper.  }For instance, it is easy to see that the code 
$$\C = \{1, 2, 123, 3, \varnothing\}$$
is intersection complete and that $\Delta(C)$ is the solid 2-simplex (otherwise known as a triangle), which is collapsible. However, the toric ideal $T_\C$ is given by 
$$T_\C = \abk{y_{123}-y_1y_2y_3}. $$
This does not have a quadratic Gr\"obner basis. 

\section{Discussion}
We showed that an inductively $k$ pierced code on $n$ neurons has a realization with spheres in $\R^{k+1}$ and with hyperplanes in $\R^n$. In general, this means our hyperplane embedding dimension can be much higher than our sphere embedding dimension. Can we bring the hyperplane embedding dimension down? By how much?

\begin{question} What is the minimal hyperplane embedding dimension of an inductively $k$-pierced code on $n$ neurons?
\end{question}

We constructed hyperplane and sphere realizations of inductively pierced codes using similar procedures. Thus, it is natural to ask about the relationship between these realizations. Is there some way to see one realization as the image of the other one, or vice versa? 
\begin{question}
Is there a continuous map from $X\subset \R^n$ to $\R^k$ which takes a hyperplane realization of an inductively pierced code to the sphere realization to the same code?
\end{question}
An answer to this question might give some insight into the relationship between hyperplane codes and sphere codes more generally.

We introduced an order $<$ on the codewords of a neural code and found that if $\C$ was an inductively pierced code, then $<$ gave a shelling order on $\Gamma(\C)$ and induced a lexicographic monomial order in which $T_\C$ has a quadratic Gr\"obner basis \revision{in all examples we were able to computationally check}.  Thus, it is natural to ask whether there is a deeper relationship between these two results. To this end, we observe that the toric ideal of a neural code has a very nice interpretation in terms of $\Gamma(\C)$.  

Let $\C$ be a neural code, $T_\C$ be its toric ideal. Let $p-q\in T_\C$. Then the monomials $p$ and $q$ each correspond to a subset of the facets of $\Gamma(\C)$. Then $p-q\in T_\C$ if and only if each on-vertex of $\Gamma(\C)$ appears in the same number of facets indexed by $p$ and facets indexed by $q$. We can strengthen this result slightly if we require the toric ideal to be homogeneous. We can force this by modifying the code by including a dummy neuron, which we label neuron 0,  in each codeword.  For instance, we replace the code $\{\varnothing, 1, 2, 12\}$ with the code $\{0, 01, 02, 012\}$. The inhomogeneous polynomial $y_1y_2-y_{12}\in T_\C$ becomes the homogeneous polynomial $y_{01}y_{02}-y_{012}y_0$. In this case, we have $p-q\in T_\C$ if and only if each vertex (on or off) of $\Gamma(\C)$ appears in the same number of facets indexed by $p$ and indexed by $q$. Because we can define the polar complex of an arbitrary hypergraph, these connections between the toric ideal of a hypergraph and the combinatorial topology of its polar complex may be more broadly interesting.

Thus, it is natural to ask whether, if $\Gamma(\C)$ is shellable,  $T_\C$ has a quadratic Gr\"obner basis induced by the shelling order? It turns out that the answer to this question is no. In a search of random nondegenerate hyperplane codes, we found the code  \begin{align*}\C = \{134, 13, 3, \varnothing, 1, 12, 34, 234, 1234, 123, 4\}.\end{align*} The polar complex $\Gamma(\C)$ is shellable because $\C$ is a nondegenerate hyperplane code. The codewords are written in a shelling order of $\Gamma(\C)$. However, the Gr\"obner basis for $T_\C$ in the lexicographic monomial order corresponding to this ordering on the codewords is 
\begin{align*}y_{00011} y_{11101} - y_{00111} y_{11001}, y_{00011} y_{00101} - y_{00111} y_{00001},\\ y_{00011} y_{10101} - y_{00001} y_{10111}, y_{11101} y_{00111} - y_{01111} y_{10101}, \\y_{11101} y_{10001} - y_{11001} y_{10101}, y_{11101} y_{00001} - y_{11001} y_{00101}, \\y_{11101} y_{10111} - y_{11111} y_{10101}, y_{11111} y_{00111} - y_{01111} y_{10111},\\ y_{11111} y_{10001} - y_{11001} y_{10111}, y_{11111} y_{00001} - y_{00111} y_{11001}, \\y_{11111} y_{00101} - y_{01111} y_{10101}, y_{01111} y_{10001} - y_{00111} y_{11001},\\ y_{01111} y_{00001} y_{10101} - y_{00111} y_{11001} y_{00101}, \\ y_{01111} y_{00001} y_{10111} - y_{00111}^{2} y_{11001}, y_{00111} y_{10001} - y_{00001} y_{10111},\\ y_{00111} y_{10101} - y_{00101} y_{10111}, y_{10001} y_{00101} - y_{00001} y_{10101}
\end{align*}
which contains the cubic terms
\begin{align*} y_{01111} y_{00001} y_{10101} - y_{00111} y_{11001} y_{00101}\end{align*}
and
\begin{align*}
y_{01111} y_{00001} y_{10111} - y_{00111}^{2} y_{11001}.
\end{align*}
However, the toric ideal of this code, and of all other random hyperplane codes we have checked, does have a quadratic Gr\"obner basis in another term order.  (For more examples and computations, see \url{https://github.com/lienkaemper/Neural-Toric-Ideals}. ) In light of this, we ask 
\begin{question}
If $\C$ is a nondegenerate hyperplane code, does $T_\C$ have a quadratic Gr\"obner basis? Does it have a quadratic generating set? 
\end{question}

More generally, how can we describe the class of neural codes whose toric ideals have quadratic generating sets? Are there other nice geometric or combinatorial properties they must all satisfy? 

In \cite{jeffs2018morphisms}, Jeffs turns the set of neural codes into a a category by defining morphisms between codes. Thus, we can ask how inductively pierced codes behave under morphisms. 
\begin{question}
Which codes are images of inductively pierced codes? Which codes are preimages of inductively pierced codes?
\end{question}

%
%
%
%

\section*{Acknowledgements}
This project began as an undergraduate thesis supervised by Mohamed Omar, who I thank for great advice both during my senior year and after graduation. I thank Carina Curto for encouraging me to keep working on this project and providing guidance while I wrote this paper. I thank Alex Kunin for many helpful discussions and for providing comments on early drafts. I thank Micah Pedrick for proofreading a few drafts. I thank Whitney Liske for pointing out the error in the first version during the poster session at the Graduate Workshop in Commutative Algebra for Women \& Mathematicians of Other Minority Genders in April 2019. I was partially funded by the grant NSF DMS-1516881.

\bibliography{mybib}
\bibliographystyle{plain}

\appendix

\section{Constructing the realization with balls \label{app:balls}}

In order to construct this realization, we make use for the following lemma:

\begin{lemma}\label{lem:circles}
For $k\geq 1$, the intersection of $m$ distinct $k$-spheres, such that the $m-$way intersection between the balls bounded is nonempty and no ball contains another,  is $k-m+1$ sphere. 
\end{lemma}

\begin{proof}[Proof of Lemma]
First, we show that the intersection of a $p$-sphere and a $q$-sphere in $\R^{p+1}$ ($p\geq q$) is a $q-1$-sphere. Choose coordinates such that all but the first $q+1$ coordinates of the $q+1$-ball enclosed by the $q$-sphere are 0.  In other words, choose coordinates such that the $q+1$-plane containing the ball is defined by the constraining the $(q+2)$-th to $p+1$-st coordinates to be zero. Now, the intersection of the $p-$sphere with the $q-$sphere lies inside this hypersurface obtained by setting all but the first  $q+1$ to 0. Thus, the intersection of the $p$-sphere and the $q$-sphere in $\R^{p+1}$ is equal to the intersection of two $q$-spheres in $\R^{q+1}$. This is a $q-1$ sphere, which we can check using coordinates. 

Now, we prove this lemma. Let $S_1, \ldots, S_m$ be $k$-spheres.  Let $T_1=S_1$  and $T_{i} = T_{i-1}\cap S_i$ for $1\leq i \leq m$. Now, $T_1$ is a $k$-sphere. By the argument above, if $T_i$ is a $p$-sphere, $T_{i+1}$ is a $p-1$ sphere. Thus, by induction, $T_m = S_1\cap \cdots \cap S_m$ is a $k-m+1$ sphere, as desired. 
 \end{proof}
 
 Now, we prove Theorem \ref{thm:real}. 
 \begin{proof}[Proof of Theorem \ref{thm:real}]
 
We prove this theorem using induction on $n$, the number of neurons in the code. We take as our inductive hypothesis a slightly stronger statement: 
if $\C$ is an inductively $k$-pierced code on $n$ neurons, there exist open $k+1$-dimensional balls $U_1, \ldots, U_n$ such that $\C = \code(U_1, \ldots, U_n)$. Further, if $\C$ is $(\lambda, \sigma, \tau)$ pierceable, there exists a point $p\in \left(\bigcap_{i\in \lambda} \partial U_i\right) \cap \mathrm{int}(B_{\sigma, \tau})$.

First, note that the only inductively $k$-pierced code on one neuron, $\{0, 1\}$,  is realized using a $k+1$-ball as the single receptive field. To check the second condition, note that  $\{0, 1\}$ is $(1, \varnothing, \varnothing), (\varnothing, 1, \varnothing)$ and $(\varnothing, \varnothing, 1)$ pierceable. Thus one can quickly check that $p\in \left(\bigcap_{i\in \lambda} \partial U_i\right) \cap \mathrm{int}(B_{\sigma, \tau})$ for all possible partitions $(\lambda, \sigma, \tau)$.  
 Now, assume the statement holds for all inductively $k$-pierced neural codes on $n-1$ neurons.  Let $\C'$ be an inductively  $k$-pierced code on $n$ neurons. Now, since $\C'$ is inductively $k$-pierced, there is some code $\C$ and a partition $\lambda \sqcup \sigma\sqcup \tau $  of $ [n-1]$, $|\lambda| =\ell \leq k$ such that $$\C' = \pierce \C\lambda\sigma \tau.$$  By the inductive hypothesis, we can construct a convex realization  $U_1, \ldots, U_{n-1}$ of $\C$ such that each $U_i$ is a  $k+1$-ball in $\R^{k+1}$. Further, we can pick a point $p\in \left(\bigcap_{i\in \lambda} \partial U_i\right) \cap \mathrm{int}(B_{\sigma, \tau})$. Let $r_i$ give the radius of $U_i$.
 
 By Lemma \ref{lem:circles}, $\bigcap_{i\in \lambda} \partial U_i$ is a $k-|\lambda|+1 \geq 1$ sphere. This, together with the fact that $p\in   \mathrm{int}(B_{\sigma, \tau})$, means that we we can pick $r$ such that the open ball $B_r(p)\subset B_{\sigma, \tau}$ and $B_{r}(p)$ does not cover the set $ \left(\bigcap_{i\in \lambda} \partial U_i\right) \cap \mathrm{int}(B_{\sigma, \tau})$. Choose $r_n <r, U_n= B_{r_n}(p)$. 
 
 By construction, $U_n\subset B_{\sigma, \tau}$, but does not cover it. Together with the construction $r_n<r_i$ for all $i<n$,  $U_n$ does not cover any atom of $\C$. Thus, every codeword of $\C$ is a codeword of $\C'$. Further, since $U_n\subset B_{\sigma, \tau}$, any codeword of $\C'$ which contains $n$ is compatible with the background motif $(\sigma, \tau)$. Finally, since $p\in \left(\bigcap_{i\in \lambda} \partial U_i\right)$, $U_n$ contains a point in $A_\nu$ for each $\nu\subset \lambda$. 

Next, we check that if $\C'$ is $(\lambda', \sigma', \tau')$ pierceable, there exists a point $p\in \left(\bigcap_{i\in \lambda'} \partial U_i\right) \cap \mathrm{int}(B_{\sigma', \tau'})$. By construction, $U_n$ did not cover the set  $\left(\bigcap_{i\in \lambda} \partial U_i\right) \cap \mathrm{int}(B_{\sigma, \tau})$. Further, by construction, $U_n$ does not meet any boundaries $\partial U_i$ for $i\notin \lambda$. Thus, it remains to check that $\left(\bigcap_{i\in \lambda'} \partial U_i\right) \cap \mathrm{int}(B_{\sigma', \tau'})$ is nonempty for $n\in \lambda'$.  To see this, first note that $\C$ is $(\lambda', \sigma', \tau')$ pieceable, $\sigma \subset \sigma'$ and $\tau \subset \tau'$. Next, by Lemma \ref{lem:circles}, we have that $\bigcap_{i\in \lambda'}$ is a $k-|\lambda|+1\geq 1$ sphere. Since any $i\in \lambda\setminus \lambda'$ pass through $\bigcap_{i\in \lambda'} \partial U_i$, we can thus find a point in any valid $\left(\bigcap_{i\in \lambda'} \partial U_i\right) \cap \mathrm{int}(B_{\sigma', \tau'})$. 
\end{proof}

\section{Constructing the hyperplane realization \label{app:hyperplanes}} 
\begin{proof}

We prove this by giving an inductive construction, illustrated in Figure \ref{fig:hyper_const}.

As a base case, note that $\{\varnothing, 1\}$ has a hyperplane realization whose bounding polytope is a line segment, as illustrated in Figure \ref{fig:hyper_const}. 

Now, let $\C'$ be an inductively pierced code on $[n]$. By definition, we can find an inductively pierced code $\C$ on $[n-1]$ such that $\C'$ is constructed from $\C$ as a piercing of $\lambda\subset [n-1]$ subject to the background motif $(\sigma, \tau)$ such that $(\lambda, \sigma, \tau)$ partitions $[n-1]$. That is, 
$$\C' = \pierce  \C \lambda \sigma \tau $$ 
By the inductive hypothesis, we can construct a hyperplane realization of $\C$ whose bounding region is a $n-1$ simplex. That is, $$\C = \code(H_1^+, \ldots, H_{n-1}^+, \Delta_{n-1}).$$ By the definition of piercing, $\sigma\cup \nu \in \C$ for each $\nu\subset \lambda$. Thus, since $(H_1^+, \ldots, H_{n-1}^+, \Delta_{n-1})$ realizes $\C$, we can find some  $$p\in \left(\bigcap_{i\in \lambda} H_i\right)\cap \left(\bigcap_{i\in \sigma} H_i^+ \right)\cap \left(\bigcap_{j\in \tau} H_j^-\right)$$

Since we picked $p$ in the interior of half spaces labeled by $\sigma$ and the interior of the complements of the half spaces labeled by $\tau$, we can find $r>0$ such that $B_r(p)\cap H_j = \varnothing $  if $j\notin \lambda$. Since $p$ is on all hyperplanes labeled by $\lambda$, any open set around $p$ intersects all $2^k$ atoms labeled with a subset of $\lambda$ on, $\sigma$ on, and $\tau$ off. Thus, we could construct a realization of $\C'$ by  taking $B_r(p)$ as the place field for neuron $n$.  However, this would not be hyperplane realization. 

Roughly, to turn this realization into a hyperplane realization, we pull the point $p$ out into an extra dimension and it  from the space containing the realization of $\C$ using a hyperplane.  More concretely, we will extend all hyperplanes up into an additional dimension, such that they are perpendicular to the plane containing the realization of $\C$. We will then construct a new bounding simplex by taking the cone over the realization of $\C$ in the plane with the point $\tilde p '$, a point located directly above a point $p'$ chosen close to $p$. Finally, we will take $H_n^+$ to be a half-space containing the point $p$ and a small neighborhood around it.  

We begin by finding the point $p'$. Our goal is to choose a point $p$' near $p$ which is not on any hyperplane, but so that there is a copy of the bounding simplex $\Delta_{n-1}'$ dilated about $p'$ which meets all $2^k$ atoms labeled with a subset of $\lambda$ on, $\sigma$ on, and $\tau$ off.   This is to avoid degeneracy between the bounding simplex and the hyperplanes. If we skipped this step, we would still construct a hyperplane realization of $\C'$, but it would be degenerate.

First, choose an dilation factor $a$ such that the copy of $\Delta_{n-1}$ scaled by $a$ about $p$ fits inside $B_{r/2}(p)$. Now, let $r'$ give the shortest distance between $p$ and a boundary point of this simplex. Then it we pick $p'$ such that $d(p, p')<r'$, a copy of $\Delta_{n-1}$ scaled by $a$ about $p'$ will still be contained in  $B_{r}(p)$ and will contain an open neighborhood of $p$. Thus this scaled simplex $\Delta_{n-1}'$ meets the desired atoms and no others.

Now, we construct a $n+1$ dimensional hyperplane realization of the same code as follows. Let the ambient space be $\R^{n+1}$, and view the copy of $\R^n$ containing the realization of $\C$ as the $x_{n+1} = 0$ plane. Now, extend each hyperplane in the realization of $\C$ by keeping its defining equation the same. Since there is no constraint on $x_{n+1}$, the resulting hyperplane is perpendicular to the $x_{n+1} = 0$ plane. Now, pick $\tilde p ' $ above $p'$ at height $1$. Let the new bounding simplex $\Delta_n$ be the cone over the old bounding simplex $\Delta_{n-1}$ with the point $\tilde p '$.

Now, take the  hyperplane $H_{n}$ parallel to the $x_{n+1} = 0$ plane at height $1-a$, where $a$ is the scale factor we used when choosing the point $p'$.  Then by similar triangles, the cross section of the new bounding polytope by the hyperplane $H_{n}$ projects down to a copy of the bounding simplex scaled towards $p'$ by a factor of $a$.  Let the half space $H_n^+$ be the half space cut off by $H_n$, oriented up. We claim that $$\code(H_1^+, \ldots, H_n^+, \Delta_{n}) = \C'.$$
 
 To see this, first note that, since the hyperplane $H_n$ is at a positive height above the $x_{n+1} = 0$ plane containing a realization of $\C$, each codeword of $\C$ is a codeword of $\code(H_1^+, \ldots, H_n^+, \Delta_{n}) $. 
 
 Now, note that, since we extended all hyperplanes perpendicular to the $x_{n+1} = 0 $ plane, if the point $q'$ is exactly above the point $q$, $q$ and $q'$ are on the same side of the hyperplanes $H_1, \ldots, H_n$. Thus, any codeword of  $\code(H_1^+, \ldots, H_n^+, \Delta_{n}) $ is either a codeword of $\C$ or a codeword of $\C$ with $n$ added to it. Now, since the points in $H_n^+$ are those points directly above the scaled simplex we chose around $p'$.  Thus 
 $$\C' = \code(H_1^+, \ldots, H_n^+, \Delta_{n}), $$
 as desired. 
\end{proof}

\end{document}